\documentclass[11pt]{article}
\usepackage{setspace,xargs}
\usepackage{amsmath}

\usepackage{graphicx,subfigure}
\usepackage{graphicx,xcolor}
\graphicspath{{./images/}}
 \usepackage{epsfig}
 \usepackage{cite}
\usepackage [a4paper, bindingoffset=-.5cm, footskip=1cm, headheight = 16pt, top=3cm, bottom=2.5cm, right=3cm, left=3cm ] {geometry}
\input{amssym}
\input{epsf}
\newtheorem{prethm}{{\bf Theorem}}

\newenvironment{theorem}{\begin{prethm}{\hspace{-0.5
em}{\bf.}}}{\end{prethm}}
\newtheorem{preex}{{\bf Example}}

\newtheorem{prelem}{{\bf Lemma}}

\newtheorem{precor}{{\bf Corollary}}

\newenvironment{corollary}{\begin{precor}{\hspace{-0.5
em}{\bf.}}}{\end{precor}}
\newtheorem{prepos}{{\bf Proposition}}

\newenvironment{proposition}{\begin{prepos}{\hspace{-0.5
em}{\bf.}}}{\end{prepos}}
\newtheorem{preobv}{{\bf Observation}}

\newtheorem{predef}{{\bf Definition}}

\newtheorem{preproof}{{\bf Proof.}}

\newenvironment{proof}[1]{\begin{preproof}{\rm
               #1}\hfill{$\rule{2mm}{2mm}$}}{\end{preproof}}

\begin{document}
\renewcommand{\baselinestretch}{1.3}
\date{}
\title{{\Large\bf Star Edge Coloring of the Cartesian Product of Graphs}}

{\small
\author{
{\sc Behnaz Omoomi} 
\ and
{\sc Marzieh Vahid Dastjerdi}
\\ [1mm]
{\small \it  Department of Mathematical Sciences}\\
{\small \it  Isfahan University of Technology} \\
{\small \it 84156-83111, \ Isfahan, Iran}}
\maketitle
\begin{abstract}
\noindent A star edge coloring of a graph $G$ is a proper edge coloring of $G$ such that  every path and cycle of length four in $G$  uses at least three different colors. The star chromatic index of a graph $G$, is the smallest integer $k$ for which $G$ admits a star edge coloring with $k$ colors.  In this paper, we first obtain some upper bounds for the star chromatic index of the Cartesian product of two graphs. We then  determine the exact value of the star chromatic index of $2$-dimensional grids.  We also obtain some upper bounds on the star chromatic index of the Cartesian product of a path with a cycle,  $d$-dimensional grids, $d$-dimensional hypercubes and  $d$-dimensional toroidal grids,  for every positive integer $d$.
\end{abstract}

\section{INTRODUCTION}
Here we briefly introduce the graph theory terminology and notations that we use in this paper. For further information on graph theory concepts we refer the reader to \cite{bondy}. All graphs considered in this paper are finite, simple and undirected. We use $P_n$ and $C_n$ to denote   a path and a cycle   of order $n$, respectively. A path (cycle) with $k$ edges is referred to as a \textit{$k$-path} (\textit{$k$-cycle}). The \textit{length} of a path (or a cycle)  is the number of its  edges. Let $G$ be a  graph with vertex set $V(G)$ and edge set $E(G)$.  The number of edges that meet a specific vertex in $G$ is called the \textit{degree} of that vertex. The \textit{maximum degree} of $G$,  denoted by $\Delta(G)$ or simply $\Delta$, is the maximum over the set of degrees of all  vertices in  $G$.   The \textit{distance} of two edges in $G$ is the minimum length of the paths between every two end-points of these edges. A subset $M$ of   edges in $G$  is called a \textit{matching}  if every two edges in $M$ have no common end-point. A matching $M$ in $G$ is a \textit{perfect matching} if every vertex of $G$  is an end-point of an edge in $M$.   \\
The \textit{Cartesian product} of two graphs $G$ and $H$, denoted by $G\square H$, is a graph with vertex set $V (G) \times V (H)$, and  $(a,x)(b,y)\in E(G\square H)$ if either $ab\in E(G)$ and $x=y$, or $xy\in E(H)$ and $a=b$. 
 The vertex set of this graph can be considered as a $|V(G)|\times |V(H)|$ array such that  the subgraph induced by each row is isomorphic to $H$ (a copy of $H$) and the subgraph induced by each column is isomorphic to $G$ (a copy of $G$). A \textit{$d$-dimensional grid} $G_{l_1,l_2,\ldots,l_d}=P_{l_1}\square P_{l_2}\square\ldots\square P_{l_d}$ is the Cartesian product of $d$ paths. A \textit{$d$-dimensional hypercube} $Q_d$ is  the Cartesian product of $P_2$ by itself $d$ times. A \textit{$d$-dimensional toroidal grid} $T_{l_1,l_2,\ldots,l_d}=C_{l_1}\square C_{l_2}\square\ldots\square C_{l_d}$ is the  Cartesian product of $d$ cycles.
\par 
 A \textit{proper vertex} (\textit{edge coloring}) of graph $G $ is an assignment of colors to the vertices (edges) of $G$ such that no two adjacent vertices (edges) receive  the same color.  The minimum number of colors that is needed to color the  vertices (edges)  of $G$ properly is called the \textit{chromatic number} (\textit{chromatic~index}) of $G$, and is denoted by $\chi(G)$ \rm{(}$\chi^\prime (G)$\rm{)}. 
A subgraph $F$ of $G$ is said to be \textit{bi-colored}, if the restriction of the proper vertex (edge) coloring of $G$ to $F$, is a vertex (edge) coloring with  at most two colors \cite{mohar}. A \textit{star vertex coloring} of  $G$, is a proper vertex coloring such that no path or cycle  on four vertices in G is bi-colored (see \rm{\cite{fertin,Timmons}}).  
\par In 2008, Liu and Deng \cite{delta} introduced the edge version of the star vertex coloring that  is defined as follows. A \textit{star~edge~coloring}   of   $G$ is a proper edge coloring of $G$ such that  no path or cycle of length four in $G$ is bi-colored. We call a star edge coloring of $G$ with $k$ colors, a \textit{$k$-star~edge~coloring} of $G$. The smallest integer $k$   for which  $G$ admits a $k$-star edge coloring is called the \textit{star~chromatic~index}  of $G$ and is denoted by  $\chi^\prime_s(G)$.  Liu and Deng \cite{delta} presented  an upper bound on the star chromatic index of graphs with maximum degree $\Delta\geq 7$. In \cite{mohar},  Dvo{\v{r}}{\'a}k et al.   obtained the  lower bound $2\Delta(1 + o(1))$ and  the near-linear upper bound $\Delta . 2^{O(1)\sqrt{\log\Delta}}$  on the star chromatic index of graphs with maximum degree $\Delta$.  They also presented some upper bounds  and lower bounds on the star chromatic index of  complete graphs  and  subcubic graphs (graphs with maximum degree at most 3).
 In \rm{\cite{class}}, Bezegov{\'a}  et al.   obtained some bounds on the star chromatic index of subcubic outerplanar graphs, trees and outerplanar graphs.
Some other results on the star chromatic index of graphs can be seen in \cite{11,12,17,18,luzar,pradeep}.
\par 
This paper is organized as follows.
 In Section~\ref{general}, we give some tight upper bounds on the star chromatic index of the Cartesian product of two arbitrary graphs $G$ and $H$.   In Section \ref{paths and cycles}, we determine the exact value of $\chi^\prime_s(P_m\square P_n)$ for every integers $m,n\geq2$. Then, we prove  that $\chi^\prime_s(P_m\square C_n)\leq 7$, and we determine the exact value of $\chi^\prime_s(P_m\square C_n)$  for infinite values of $m$ and $n$. Finally, we give some upper bounds for the star chromatic index of the Cartesian product of two cycles. Moreover, applying the upper bounds that we obtained in Section \ref{general}, we give upper bounds on the star chromatic index of $d$-dimensional grids, $d$-dimensional hypercubes,  and $d$-dimensional toroidal grids.

\section{GENERAL UPPER BOUNDS}\label{general}
In this section, we first give an upper bound on the star chromatic index of the Cartesian product of two graphs in terms of their star chromatic indices and their chromatic numbers. Then, we define the concept of star compatibility  and use this  concept to obtain another upper bound for the star chromatic index of the Cartesian product of two graphs. Naturally, these bounds  imply some upper bounds on the star chromatic  index of $G\square P_n$ and $G\square C_n$, where  $G$ is an arbitrary graph. \par 
For proving some of the results in this section, we first need following  fact about the star chromatic index of cycles. It is derived from proof of Theorem~5.1 in~\cite{mohar}.
\begin{proposition}\label{p1} \rm{\cite{mohar}}
If $n\geq 3$ is a positive integer, then we have 
\[\chi^\prime_s(C_n)=
\begin{cases}
     3 &\quad\text{if}\hspace*{2mm} n\neq 5,\\
     4 &\quad\text{if}\hspace*{2mm} n=5.
         \end{cases}
\]
\end{proposition}
\begin{theorem}\label{th2}
For every two graphs $G$ and $H$, we have 
\begin{equation*}
\chi^\prime_s(G\square H)\leq \min \lbrace\chi^\prime_s(G)\chi(H)+\chi^\prime_s(H),\chi^\prime_s(H)\chi(G)+\chi^\prime_s(G)\rbrace.
\end{equation*}
Moreover, this bound is tight.
\end{theorem}
\begin{proof}{
Let $c_G$ be a proper vertex coloring of $G$ by  colors $\lbrace 0,1,2,\ldots,\chi(G)-1\rbrace$ and  $c_H$ be a proper vertex coloring of $H$ by  colors $\lbrace 0,1,2,\ldots,\chi(H)-1\rbrace$. Also, let $f_G:E(G)\rightarrow\lbrace 0,1,2,\ldots,\chi^\prime_s(G)-1\rbrace$ be a star edge coloring of  $G$ and $f_H:\nolinebreak E(H)\rightarrow\lbrace \chi^\prime_s(G),\chi^\prime_s(G)+1,\chi^\prime_s(G)+2,\ldots,\chi^\prime_s(G)+\chi^\prime_s(H) -1 \rbrace$ be a star edge coloring of  $H$.  We define edge coloring $f$ of $G\square H$ as follows. For every $x\in V(H)$ and $ab\in E(G)$, let
\begin{eqnarray*}
 f((a,x)(b,x))=f_G(ab)+\chi^\prime_s(G)c_H(x).
\end{eqnarray*}
For every $a\in V(G)$ and $xy\in E(H)$, let
\begin{eqnarray*}
 f((a,x)(a,y))=f_H(xy).
\end{eqnarray*}
Since edge colorings $f_G$ and $f_H$ use disjoint color sets, the edge coloring $f$ is proper. It remains to show that there is no bi-colored 4-path (4-cycle) in $G\square H$. By contrary, suppose that there is a bi-colored 4-path (4-cycle) $P$ in  $G\square H$. Since star edge colorings  $f_G$ and $f_H$  use different color sets,  the edges of $P$ belong alternately to $G$ and  $H$; otherwise  $P$ is not bi-colored.  Assume that  $e=(a,x)(b,x)$ and $e^\prime=(b,y)(c,y)$ are the edges of $P$ that correspond to edges $ab$ and $bc$ in  $G$.   Since  $xy\in E(H)$, then  $c_H(x)\neq c_H(y)$. Hence, we have 
\begin{equation*}
f((a,x)(b,x))-f((b,y)(c,y))=f_G(ab)-f_G(bc)+\chi^\prime_s(G)(c_H(x)-c_H(y))\neq 0.
\end{equation*} 
Therefore, path (or cycle) $P$ uses at least three different colors,  which is  a contradiction. Thus, edge coloring $f$ is a star edge coloring with  $\chi^\prime_s(G)\chi(H)+\chi^\prime_s(H)$ colors.\\
Similarly, we can define  edge coloring $g$ as follows. For every $x\in V(H)$ and $ab\in E(G)$, let
\begin{eqnarray*}
 g((a,x)(b,x))=f_G(ab).
\end{eqnarray*}
For every $a\in V(G)$ and $xy\in E(H)$, let
\begin{eqnarray*}
  g((a,x)(a,y))=f_H(xy)+\chi^\prime_s(H)c_G(a).
\end{eqnarray*}
By a similar argument,  $g$ is also a star edge coloring with  $\chi^\prime_s(H)\chi(G)+\chi^\prime_s(G)$ colors. Therefore, we have
\begin{equation*}
\chi^\prime_s(G\square H)\leq \min \lbrace\chi^\prime_s(G)\chi(H)+\chi^\prime_s(H),\chi^\prime_s(H)\chi(G)+\chi^\prime_s(G)\rbrace.
\end{equation*}
Now, we prove the tightness of the bound. Let $G\cong P_2$ and $H\cong C_{2n}$, where $n$ is a positive integer. By  Proposition~\ref{p1}, $\chi^\prime_s(C_{2n})=3$. Thus, we have 
$$\chi^\prime_s(G\square H)\leq \chi^\prime_s(P_2)\chi(C_{2n})+\chi^\prime_s(C_{2n})=5.$$
In Theorem~\ref{th8},  we will show that  $\chi^\prime_s(G\square H)=5$, which implies  the tightness of our bound.
}\end{proof}
Note that in  proof of Theorem \ref{th2}, we use $\chi(H)$ (or $\chi(G)$)  different star edge colorings of $G$ (or $H$)  to avoid  bi-colored 4-paths (4-cycles) in $G\square H$.  This idea motivates us to  define the concept of star compatibility, in order to improve the upper bound on $\chi^\prime_s(G\square H)$ as follows.
\par Let $f$ be a star edge coloring of a graph  $G$. For every vertex $v\in V(G)$, we denote  the set of  colors of  edges incident to $v$ by $A_f(v)$. Two star edge colorings $f_1$ and $f_2$ of  $G$ are called \textit{star compatible} if for every vertex $v$, $A_{f_1} (v) \cap A_{f_2} (v) =\emptyset$.
We say that  graph $G$ is \textit{$(k, t)$-star colorable} if $G$ has $t$ pairwise star compatible colorings  $f_i : E(G) \rightarrow\lbrace 0, 1,\ldots , k-1\rbrace$, $1 \leq i \leq t$.
\par 
Now, in the following theorem, we use star compatibility to  find another upper bound on the star chromatic index of the Cartesian product of graphs.
\begin{theorem}\label{th:col}
If $G$ and $H$ are two graphs such that $G$ is $(k_G, t_G)$-star colorable and  $t_G\geq\chi(H)$, then 
\begin{equation*}
\chi^\prime_s(G\square H)\leq k_G +\chi^\prime_s(H).
\end{equation*}
\end{theorem}
\begin{proof}
{
Let   $f_i : E(G) \rightarrow \lbrace 0, 1, \ldots , k_G-1\rbrace$, $0\leq i\leq t_G-1$, be star compatible colorings of $G$ and $f_H:E(H)\rightarrow \lbrace k_G,k_G+1,\ldots, k_G+\chi^\prime_s(H)-1\rbrace$  be a star edge coloring of $H$. Also, let $c_H$ be a proper vertex coloring of $H$ using  colors $\lbrace 0, 1, \dots, \chi(H)-1\rbrace$.   We define edge coloring $f$ of $G\square H$ as follows. For every edge $ab$ of $G$ and every vertex $x$ of $H$,  let
\begin{eqnarray*}
 f((a,x)(b,x))=f_{c_H(x)}(ab).
\end{eqnarray*} 
 For every edge $xy$ of $H$ and every vertex $a$ of $G$,  let
\begin{eqnarray*}
f((a,x)(a,y))=f_H(xy).
\end{eqnarray*}
Note that edge coloring $f$ uses  $k_G+\chi^\prime_s(H)$ different colors. 
Since the colors of edges in $G$ and $H$ are different and colorings $f_H$ and $f_i$, $0\leq i\leq t_G-1$, are star edge colorings, then the edge coloring $f$ is a proper edge coloring and every path (or cycle) with two adjacent edges in $G$ or $H$ is not bi-colored. Thus,  we only need to consider the  case where we have a path  (or cycle) with edges $e_1=(a,x)(b,x)$, $e_2=(b,x)(b,y)$ and $e_3=(b,y)(c,y)$, respectively. In such a case, we have $f(e_1)=f_{c_H(x)}(ab)$ and $f(e_3)=f_{c_H(y)}(bc)$. On the other hand, since $xy\in E(H)$, then $c_H(x)\neq c_H(y)$, and consequently star edge colorings $f_{c_H(x)}$ and  $f_{c_H(y)}$ are star compatible. Therefore, $f_{c_H(x)}(ab)\neq f_{c_H(y)}(bc)$. This shows that $f$ is a star edge coloring of $G\square H$.
 }\end{proof}
Note that if $G$ is a  $(k_G,t_G)$-star colorable graph, then  it is also $(ak_G,at_G)$-star colorable for every positive integer $a$. In particular,
 every graph $G$ is $(a\chi^\prime_s(G),a)$-star colorable and therefore, for every graph $H$, graph $G$ is $( \chi^\prime_s(G)\chi(H),\chi(H))$-star colorable. Thus, by Theorem \ref{th:col}, we have 
 \begin{equation*}
 \chi^\prime_s(G\square H)\leq\nolinebreak \chi^\prime_s(G)\chi(H)+\chi^\prime_s(H).
 \end{equation*} 
 Hence, the upper bound in Theorem \ref{th2} can be also  obtained from Theorem~\ref{th:col}.
\begin{theorem}\label{th:col2}
If graphs $G$ and $H$  are $(k_G, t_G)$-star colorable  and   $(k_H, t_H)$-star colorable, respectively, where $t_G\geq\chi(H)$ and $t_H\geq\chi(G)$, then $G\square H$ is $(k_G +k_H,\min\lbrace t_G,t_H\rbrace)$-star colorable.
\end{theorem}
\begin{proof}{
Assume that  $g_i:E(G)\rightarrow\lbrace 0,1,\ldots,k_G-1\rbrace$, $0\leq i\leq t_G-1$, are star compatible colorings of $G$ and $h_i:E(H)\rightarrow\lbrace k_G,k_G+1,\ldots,k_G+k_H-1\rbrace$, $0\leq i\leq t_H-1$, are star compatible colorings of $H$. Also, let $c_G:V(G)\rightarrow\lbrace0,1,\ldots,\chi(G)-1\rbrace$ be a proper vertex coloring of $G$ and $c_H:V(H)\rightarrow\lbrace0,1,\ldots,\chi(H)-1\rbrace$ be a proper vertex coloring of $H$. If $t=\min \lbrace t_G,t_H\rbrace$, then for each $i$,  $0\leq i\leq t-1$, we define edge coloring $f_i$ of $G\square H$ as follows.
For every edge $ab$ of $G$ and for every vertex $x$ of $H$, let $m_i(x) = (c_H(x) + i) \hspace*{-1mm} \pmod {t_G}$ and 
\begin{eqnarray*}
 f_i((a,x)(b,x))=g_{m_i(x)}(ab).
\end{eqnarray*}
Also, for every edge $xy$ of $H$ and for every vertex $a$ of $G$, let $n_i(a) = (c_G(a) + i) \hspace*{-1mm} \pmod {t_H}$ and 
\begin{eqnarray*}
 f_i((a,x)(a,y))=h_{n_i(a)}(xy).
\end{eqnarray*}
By proof of  Theorem \ref{th:col}, $f_i$'s are star edge colorings of $G\square H$. It suffices to prove that $f_i$'s  are pairwise compatible as follows. For each vertex $(a,x)$, let consider colorings $f_i$ and $f_j$, where $0\leq i<j\leq t-1$. We can easily see that 
$$A_{f_i}((a,x))=A_{g_{m_i(x)}}(a)\cup A_{h_{n_i(a)}}(x),$$
and  $$A_{f_j}((a,x))=A_{g_{m_j(x)}}(a)\cup A_{h_{n_j(a)}}(x).$$ 
Also, for $i\neq j$, we have $m_i(x)\neq m_j(x)$ and $n_i(a)\neq n_j(a)$. Therefore, $A_{g_{m_i(x)}} (a) \cap A_{g_{m_j(x)}}(a)=\emptyset$ and  $A_{h_{n_i(a)}} (x) \cap A_{h_{n_j(a)}}(x)=\emptyset$. Moreover, since the star edge colorings of $G$ and $H$ use  different sets of colors,  then $A_{g_{m_i(x)}} (a) \cap A_{h_{n_j(a)}}(x)=\emptyset$ and $A_{g_{m_j(x)}} (a) \cap A_{h_{n_i(a)}}(x)=\emptyset$. Thus, we conclude that 
$A_{f_i}((a,x))\cap A_{f_j}((a,x))=\emptyset$, as desired.
}
\end{proof}
\begin{corollary}\label{cor:Gd}
If $G$ is a $(k_G, t_G)$-star colorable graph, where $t_G\geq \chi(G)$, then $G_d=G\square G\square\ldots\square G$ \rm{($d$ times)} is $(d  k_G, t_G)$-star colorable.
\end{corollary}
\begin{proof}
{If $d=2$, then by Theorem \ref{th:col2}, graph $G_2$ is $(2k_G, t_G)$-star colorable. Now, suppose that $d>2$ and  $G_{d-1}$ is $((d-1)k_G, t_G)$-star colorable. Then using  Theorem \ref{th:col2} and by induction on $d$, we conclude that $G_d=G_{d-1}\square G$ is $(d  k_G, t_G)$-star colorable.
}
\end{proof}
In order to study the star chromatic index of the Cartesian product of paths and cycles with an arbitrary graph, in the following theorems we present some star compatible colorings of paths and cycles.
\begin{theorem}\label{per:path}
For every integers $n,r\geq2$, path $P_n$ is $(2r, r)$-star colorable.
\end{theorem}
\begin{proof}
{
Let $V(P_n)=\lbrace 0,1,\ldots,n-1\rbrace$ and $E(P_n)=\lbrace xy: 0\leq x\leq n-2, y=x+1\rbrace$. We define  edge colorings $f_i:E(P_n)\rightarrow\lbrace 0,1,\ldots,2r-1\rbrace$, $0\leq i\leq r-1$, as follows. For each $xy\in E(P_n)$  with  $x<y$, let
$$f_i(xy)=x+2i\hspace*{-1.5mm} \pmod{2r}.$$
 Clearly $f_i$'s are star edge colorings of $P_n$. For every $i$ and $j$, where $0\leq i<j\leq r-1$, and every $x\in V(P_n)$, we have
  \[
A_ {f_{i}}(x)\cap A_ {f_{j}}(x)  = 
     \begin{cases}
     \lbrace 2i \rbrace\cap \lbrace  2j  \rbrace  &\quad\text{if}\hspace*{2mm} x=0,\\
     \lbrace x-1+2i , x+2i \rbrace\cap \lbrace x-1+2j , x+2j  \rbrace  &\quad\text{if}\hspace*{2mm} 0<x<n-1,\\
      \lbrace n-2+2i \rbrace\cap \lbrace  n-2+2j  \rbrace  &\quad\text{if}\hspace*{2mm} x=n-1.
         \end{cases}
\] 
Since $0\leq i<j\leq r-1$, then  $2i\neq 2j\hspace*{-1mm} \pmod{2r}$. Therefore, $A_ {f_{i}}(x)\cap A_ {f_{j}}(x)  =\emptyset$, and  consequently edge colorings  $f_i$ and $f_j$ are pairwise star compatible.
}
\end{proof}

\begin{theorem}\label{per:cycle}
For every integers $n,r\geq2$, we have the following statements.
\begin{enumerate}
\rm{
\item[(i)]\label{per:cyclei} If  $n\geq 4$ is even, then  $C_n$ is $(2r, r)$-star colorable.
\item[(ii)] If $n\geq 2r+1$ is odd,  then $C_n$ is $(2r+1, r)$-star colorable.
\item[(iii)] If $n\geq 3$ is odd, then  $C_n$ is $(2r+\left\lceil \frac{2r}{n-1}\right\rceil, r)$-star colorable.
}
\end{enumerate}
\end{theorem}
\begin{proof}
{
 Let  $V(C_n)=\lbrace 0,1,\ldots,n-1\rbrace$ and $E(C_n)=\lbrace xy: 0\leq x\leq n-1, y= x+1 \hspace*{-1mm} \pmod{n}\rbrace$.
\par 
\noindent \rm{(i)} Let $n\geq 4$ be an even integer. We first  consider  the case $r=2$.  Since $n$ is even, either $n=0 \hspace*{-1mm} \pmod{4}$ or $n=2 \hspace*{-1mm} \pmod{4}$. If $n=0\hspace*{-1mm} \pmod4$, then we provide two  edge colorings $f_0$ and $f_1$ of $C_n$ with the following patterns.
\begin{align*}
f_0:\nolinebreak0,1,2,3,0,1,2,3,\ldots0,1,2,3,\qquad f_1:2,3,0,1,2,3,0,1\ldots2,3,0,1.
\end{align*} 
 If $n=2\hspace*{-1mm} \pmod4$, then we provide two edge colorings $f_0$ and $f_1$ of $C_n$ with the following patterns.
 \begin{align*}
 f_0:0,1,2,3,0,1,2,3,\ldots,0,1,2,3,2,1,\qquad f_1:2,3,0,1,2,3,0,1,\ldots,2,3,0,1,0,3.
 \end{align*} 
  It is then easy to see that edge colorings $f_0$ and $f_1$ are star compatible in both cases of $n$.

 Now, we consider the case $r>2$.  We define edge colorings $f_i:E(C_n)\rightarrow\lbrace 0,1,\ldots,2r-1\rbrace$, $0\leq i\leq r-1$, as follows. For every $xy\in E(C_n)$, let
\[ f_i(xy)=
 \begin{cases}
 x+2i \hspace*{-1.5mm} \pmod {2r}&\quad \text{if}\hspace*{2mm}  0\leq x\leq n-2,\\
  n+1+2i\hspace*{-1.5mm} \pmod {2r}&\quad \text{if}\hspace*{2mm}   x= n-1.\
 \end{cases}
\]
If cycle $C_n$ under edge coloring $f_i$ has a bi-colored 4-path (4-cycle) $P:v_1v_2v_3v_4v_5$, then we must have either  $v_1=n-3$, $v_2=n-2$, $v_3=n-1$, $v_4=0$, $v_5=1$, or $v_1=n-2$, $v_2=n-1$, $v_3=0$, $v_4=1$, $v_5=2$ (the other cases are clear). In the first case, $f_i(v_1v_2)=n-3+2i\hspace*{-1mm}\pmod{2r}$, and $f_i(v_3v_4)=n+1+2i\hspace*{-1mm}\pmod{2r}$. Since $2r\geq 6$, then  $n-3\neq n+1\hspace*{-1mm} \pmod{2r}$. Therefore, $f_i(v_1v_2)\neq f_i(v_3v_4)$ and consequently $P$ is not bi-colored, which is a contradiction. In the second case,  $f_i(v_1v_2)=f_i(v_3v_4)$, and $ f_i(v_2v_3)=f_i(v_4v_5)$. It implies that $n-2+2i=2i\hspace*{-1mm} \pmod{2r}$, and $ n+1+2i=1+2i\hspace*{-1mm} \pmod{2r}$, which is a contradiction. Thus,  for every $i$, $0\leq i\leq r-1$, $f_i$ is a star edge coloring of $C_n$.
\par 
We now  show that these star edge colorings are pairwise star compatible. For every  $i$ and $j$, where $0\leq i<j\leq r-1$, and every $x\in V(P_n)$, we have
  \[
A_ {f_{i}}(x)\cap A_ {f_{j}}(x)  = 
     \begin{cases}
          \lbrace  n+1+2i,2i \rbrace\cap \lbrace  n+1+2j,2j  \rbrace  &\quad\text{if}\hspace*{2mm}  x= 0,\\
     \lbrace x-1+2i,x+2i \rbrace\cap \lbrace x-1+2j,x+2j \rbrace  &\quad\text{if}\hspace*{2mm}  1\leq x\leq n-2,\\
     \lbrace n-2+2i , n+1+2i \rbrace\cap \lbrace n-2+2j , n+1+2j  \rbrace  &\quad\text{if}\hspace*{2mm}  x= n-1.
         \end{cases}
\] 
Since $0\leq i< j\leq r-1$ and $n$ is an even integer, then $2i\neq 2j\hspace*{-1mm} \pmod{2r}$ and $n+1\hspace*{-1mm}\pmod {2r}$ is odd. Therefore, we conclude that $A_ {f_{i}}(x)\cap A_ {f_{j}}(x)  =\emptyset$, as desired.
\par 
\noindent \rm{(ii)} Let $n\geq 2r+1$ be an odd integer.  First, we consider the case $r=2$. Since $n$ is odd, either $n= 1\hspace{-1mm}\pmod{4}$, or $n= 3\hspace{-1mm}\pmod{4}$. If  $n= 1\hspace{-1mm}\pmod{4}$, then we provide two edge colorings $f_0$ and $f_1$ of $C_n$ with the following patterns.
\begin{align*}
f_0:\nolinebreak0,1,2,3,0,1,2,3,\ldots,0,1,2,3,0,1,2,4,3,\qquad f_1:4,3,0,1,2,3,0,1,2,\ldots,3,0,1,2.
\end{align*}  
 If $n= 3\hspace*{-1mm} \pmod 4$, then we provide two edge colorings $f_0$ and $f_1$ of $C_n$ with the following patterns.
 \begin{align*}
 f_0:0,1,2,3,0,1,2,3,\ldots,0,1,2,3,0,4,2, \qquad f_1:4,3,0,1,2,3,0,1,2,3\ldots,0,1,2,3,1.
 \end{align*}   
It is then easy to see that in each case edge colorings $f_0$ and $f_1$ are star compatible.

Now, we  consider the case $r>2$. Assume that $n-1 = 2rp + u$, where $u\in\nolinebreak\lbrace0,2,4,\ldots,2(r-1)\rbrace$ and $p\geq1$. Let
 $$\displaystyle b=\displaystyle\nolinebreak\frac{n-1-2r}{2}=\displaystyle\nolinebreak\frac{2r(p-1)+u}{2}.$$
  For $0\leq i\leq r-1$, we define   ordered $(b+1)$-tuples $T_i$ (each entry in the tuples represents a color), as follows.  For $0\leq i\leq r-2$, let
 $$T_i=(2r-2i-2,2r-2i-1,\ldots,2r,1,2,\ldots,2r,1,2,\ldots,2r,\ldots).$$
 Also, let define
 $$T_{r-1}=(2r,1,2,\ldots,2r,1,2,\ldots,2r,\ldots).$$

 We denote the $l$-th entry of $T_i$  by $T_i^l$.  For every $i$, $0\leq i\leq r-1$,  we provide edge coloring $f_i$ of $C_n$ with the following pattern. Let 
 $$f_0: T_0^{b},T_0^{b-1},\ldots,T_0^1,a_0,a_{1},\ldots,a_{2r-1},T_{r-1}^1,T_{r-1}^2,\dots T_{r-1}^{b},q_0.$$
 For $1\leq i\leq r-1$, let
 $$f_i: T_i^{b},T_i^{b-1},\ldots,T_i^1,a_i,a_{i+1},\ldots,a_{i+2r-1},T_{i-1}^1,T_{i-1}^2,\dots T_{i-1}^{b},q_i.$$
%
In this pattern, $a_i=(2r-1)i,a_{i+1}=(2r-1)i+1,\ldots,a_{i+2r-1}=(2r-1)(i+1)$ (arithmetics are done modulo  $2r+1$) and  $q_i$ is determined as follows, where $0\leq i\leq r-1$.
\par 
 If $b=0$, then $q_i$ is $a_{i+2r}=(2r-1)(i+1)+1\hspace*{-1mm} \pmod{2r+1}$.  If   $b>1$, then $q_i$ is  $T_i^{b+1}$. If $b=1$, then $q_i$ is adjacent to $T_{i-1}^1$ and $T_{i}^1$. Since  $T_{i-1}^1$ and $T_{i}^1$  are even, it is reasonable to choose  $q_i$ from $S_i=\lbrace 1,3,\ldots,2r-1\rbrace\setminus \lbrace a_i,a_{i+2r-1}\rbrace$. 
Now, to determine  the value of $q_i$, we describe a  bipartite graph $G(X,Y)$, as follows. Let   $X=\lbrace S_0,S_1,\ldots,S_{r-1}\rbrace$, $Y=\lbrace1,3,\ldots,2r-1\rbrace$, and vertex  $S_\alpha$ is adjacent to
vertex $s_\beta$ if and only if $s_\beta\in S_\alpha$; except $S_0$ to 1. Note that $G(X,Y)$ is a $(r-2)$-regular bipartite graph and therefore has a perfect matching.  Then, in a perfect matching of $G(X,Y)$, we take the label of the vertex that is matched to $S_i$ as $q_i$.
\par
It is easy to see that $f_i$'s are proper edge colorings of $C_n$. To prove that for every $i$, $0\leq i\leq r-1$, $f_i$ is a star edge coloring, it suffices to show that every 4-path $P$ is not bi-colored. For this purpose, we only consider the following cases for $P$. In other cases, it is easy to see that $P$ is not bi-colored.
\par If $P$ is  bi-colored  with colors   $T_i^2,T_i^1,a_i,a_{i+1}$, or $q_i,T_i^1,a_i,a_{i+1}$, then $a_{i+1}=T_i^1$. Therefore, $(2r-1)i+1=2r-2i-2\hspace*{-1mm} \pmod{2r+1}$. Hence, $(2r+1)i=2r-3\hspace*{-1mm} \pmod{2r+1}$, which is a contradiction.
\par If $P$ is  bi-colored  with colors  $a_{i+2r-2},a_{i+2r-1},T_{i-1}^1,T_{i-1}^2$, or $a_{i+2r-2},a_{i+2r-1},T_{i-1}^1,q_i$, then $a_{i+2r-2}=T_{i-1}^1$. Therefore, $(2r-\nolinebreak1)(i+1)-1=2r-2(i-1)-2\hspace*{-1mm} \pmod{2r+1}$. Hence,  $(2r+1)(i+1)=2\hspace*{-1mm} \pmod{2r+1}$, which is a contradiction.
 \par If $P$ is  bi-colored  with colors  $a_{i+2r-1},T_{i-1}^{1},q_i,T_i^1$, or $T_{i-1}^{1},q_i,T_i^1,a_i$, then $T_{i-1}^1=T_i^1$. Therefore,  $2r-2(i-\nolinebreak1)-2= 2r-2i-2\hspace*{-1mm} \pmod{2r+1}$, which is a contradiction.
 \par If $P$ is  bi-colored  with colors $a_{i+2r-1},q_i,a_i,a_{i+1}$, or $a_{i+2r-2},a_{i+2r-1},q_i,a_i$, then $a_{i+2r-1}=a_i$. Therefore,  $(2r-\nolinebreak1)(i+1)= (2r-1)i\hspace*{-1mm} \pmod{2r+1}$. Hence, $2r-1=0\hspace*{-1mm} \pmod{2r+1}$, which is a contradiction.
 \par 
  If $b>1$ and  $P$ is  bi-colored   with colors $T_{i-1}^b,q_i,T_i^b,T_i^{b-1}$, or  $T_{i-1}^{b-1},T_{i-1}^{b},q_i,T_i^b$, then $T_{i-1}^{b}=T_i^b$, which is a contradiction.

  \par 
  Thus, for every $i$, where $0\leq i\leq r-1$, $f_i$ is a star edge coloring. Now,  we  show that  $f_i$'s, are pairwise  star compatible. Let $c=n-b$, and for every vertex $x\in V(C_n)$, $d_x=x-b$. Thus, for every integers  $i$ and $j$, where   $0 \leq i<j\leq r-1$, we have
\[
A_ {f_{i}}(x)\cap A_ {f_{j}}(x)  = 
     \begin{cases}
     \vspace*{.5mm}
     \lbrace a_{i},a_{i+2r}  \rbrace\cap \lbrace a_{j},a_{j+2r}   \rbrace  &\quad\text{if}\hspace*{2mm} x=0,b=0\\ 
     \lbrace T_{i}^{b},q_{i}  \rbrace\cap \lbrace T_{j}^{b},q_{j}  \rbrace  &\quad\text{if}\hspace*{2mm} x=0,b>0,\\ 
      \lbrace T_{i}^{b-x},T_{i}^{b-(x-1)}\rbrace\cap \lbrace T_{j}^{b-x},T_{j}^{b-(x-1)}\rbrace &\quad\text{if}\hspace*{2mm}  0<x<b, \\
     \lbrace T_{i}^1,a_{i}  \rbrace\cap  \lbrace T_{j}^1,a_{j}  \rbrace &\quad\text{if}\hspace*{2mm} x=b,b>0, \\
       \lbrace a_{i+d_x-1}, a_{i+d_x}\rbrace\cap  \lbrace a_{j+d_x-1}, a_{j+d_x} \rbrace &\quad\text{if}\hspace*{2mm} b<x<c-1, \\
        \lbrace a_{i+2r-1},q_i\rbrace\cap  \lbrace a_{j+2r-1},q_j\rbrace &\quad\text{if}\hspace*{2mm}  x= c-1,b=0, \\
       \lbrace a_{i+2r-1},T_{i-1}^1\rbrace\cap  \lbrace a_{j+2r-1},T_{j-1}^1\rbrace &\quad\text{if}\hspace*{2mm}  x= c-1,b>0, \\
        \lbrace T_{i-1}^{x-c+1},T_{i-1}^{x-c+2}\rbrace\cap \lbrace T_{j-1}^{x-c+1},T_{j-1}^{x-c+2}\rbrace  &\quad\text{if}\hspace*{2mm}  c\leq x<n-1,b>1, \\
         \lbrace T_{i-1}^{b},q_i\rbrace\cap \lbrace T_{j-1}^{b},q_j\rbrace  &\quad\text{if}\hspace*{2mm}  x=n-1,b>0. \
     \end{cases}
\] 
For every $s$,  $0\leq s\leq r-1$, it is obvious that $|T_{i}^s-T_{j}^s|\geq 2$ is even and $|a_{i}-a_{j}|\geq2$. Thus, we have $A_ {f_{i}}(x)\cap A_ {f_{j}}(x)=\emptyset$.    As an example, in Figure \ref{comp1}, three pairwise star compatible colorings of  $C_{15}$ by seven colors are shown. Here, $T_0=(4,5,6,1,2)$, $T_1=(2,3,4,5,6)$, $T_2=(6,1,2,3,4)$, and $b=4$, $p=2$, and $u=2$.
\begin{figure}[!ht]
\begin{center}
\includegraphics[scale=.7]{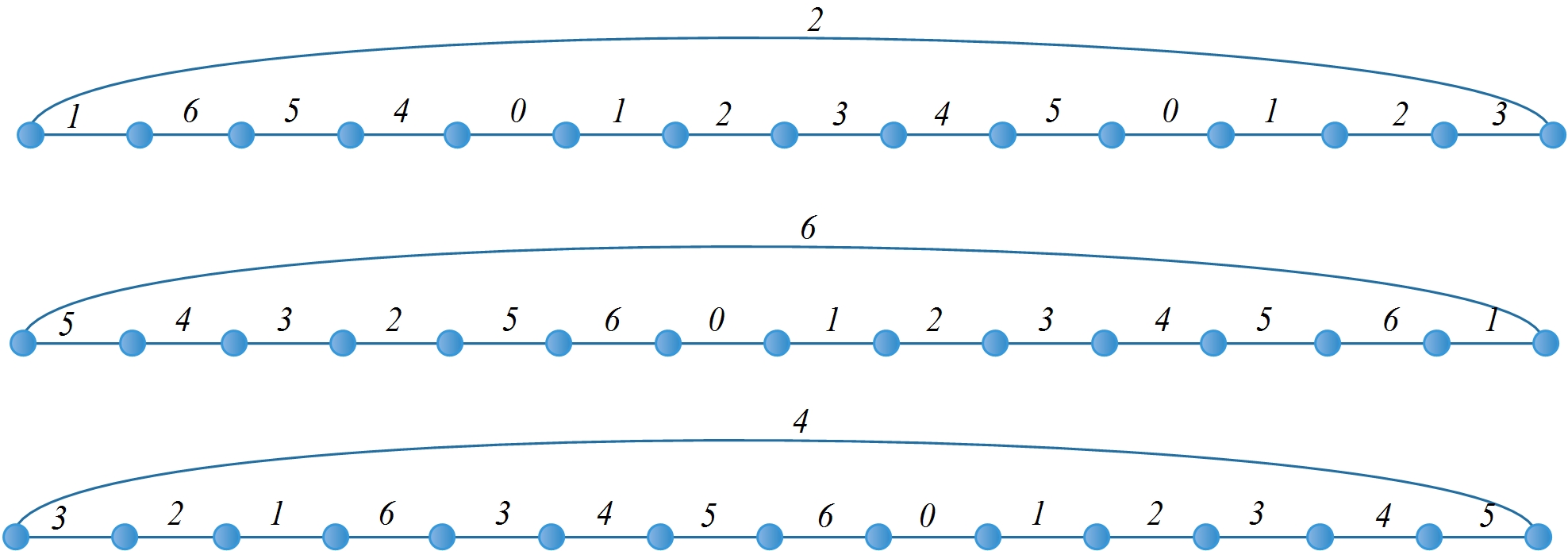}\\
\caption{Three compatible star edge colorings of  $C_{15}$.}
\label{comp1}
\end{center}
\end{figure}  
\par 
\noindent\rm{(iii)}
If $n\geq 2r+1$, then by assertion~(ii) we are done. Thus, let $3\leq n=2p+1<2r+1$ and $a=\left \lceil\frac{2r}{n-1}\right\rceil-1$. We show that $C_n$ is $(2r+1+a,r)$-star colorable. By applying assertion~\rm{(ii)}, $a$ times, we can provide  $ap$ pairwise star compatible colorings of $C_n$ with $(2p+1)a$ colors. Note that each set of  $p$ pairwise star compatible colorings uses $2p+1$ new colors.  Since 
$$2r+1+a-a(2p+1)=2(r-ap)+1\leq n,$$
by assertion~\rm{(ii)},  we can present $(r-ap)$  pairwise star compatible colorings with $2(r-ap)+1$ colors. Therefore,  we provide $ap+(r-ap)=r$ pairwise star compatible colorings of $C_n$ with $(2p+1)a+2(r-ap)+1=2r+a+1$ colors, as desired.
}
\end{proof}

By Theorem \rm{\ref{th:col}} and \rm{\ref{per:cycle}}, we have the following corollary.
\begin{corollary}\label{cor:comp}
For every graph $G$ and a positive integer $n$, we have the following statements.
\begin{enumerate}
\item[\rm{(i)}]\label{i}  If $n\geq2$, then $\chi^\prime_s(G\square P_n)\leq \chi^\prime_s( G\square C_{2n})\leq  \chi^\prime_s(G)+2\chi (G)$.
\item[\rm{(ii)}] \label{ii} If $n\geq 2\chi(G)+1$ is odd, then $\chi^\prime_s(G\square C_n)\leq \chi^\prime_s(G)+2\chi (G)+1$.
\item[\rm{(iii)}] \label{iii} If $n\geq 3$ is odd, then $\chi^\prime_s(G\square C_n)\leq \chi^\prime_s(G)+2\chi (G)+\left\lceil \frac{2\chi(G)}{n-1}\right\rceil\leq\chi^\prime_s(G)+2\chi (G)+3$.
\end{enumerate}
\end{corollary}

\section{CARTESIAN PRODUCT OF PATHS AND CYCLES}\label{paths and cycles}
 In this section, we  study the star chromatic index of grids, hypercubes, and toroidal grids. 
First, in Theorem~\ref{th4}, we obtain the star chromatic index of 2-dimensional grids and then, we  extend this result in order to get an upper  bound on the star chromatic index of $d$-dimensional grids, where $d\geq 3$. 
Also, we determine the star chromatic index of   $P_m\square C_n$    for some cases and present some upper bounds  for the rest of the cases. Applying these upper bounds, we obtain upper bounds for star chromatic index of $d$-dimensional hypercubes and $d$-dimensional toroidal grids.
\begin{theorem}\label{th4}
For two paths $P_m$ and $P_n$ with  $m,n\geq 2$, we have
\[   
\chi^\prime_s(P_m\square P_n) = 
     \begin{cases}
       3 &\quad\text{if}\hspace*{2mm} m=n=2,\\
      4 &\quad\text{if}\hspace*{2mm} (m=2,n\geq3)\hspace*{2mm}\text{or}\hspace{2mm} (m\geq3,n=2), \\
       5 &\quad\text{if}\hspace*{2mm} (m\in\lbrace3,4\rbrace, n=3)\hspace*{2mm}\text{or}\hspace{2mm}( m=3, n\in\lbrace3,4\rbrace),\\
       6 &\quad\text{otherwise.} \ 
     \end{cases}
\]
\end{theorem}
\begin{proof}{ Let $V(P_m\square P_n)=\lbrace (i,j):0\leq i\leq m-1, 0\leq j\leq n-1\rbrace$. If $m=n=2$, then $\chi^\prime_s(P_m\square P_n)=\chi^\prime_s(C_4)=3$. By symmetry, we  consider the following cases.
\par 
 \noindent\textbf{Case 1: }  $m=2$ and $n\geq 3$.\\
  It is not difficult to see that there is no 3-star edge coloring of $P_2\square P_3$. Hence, for every $n\geq 3$, $ \chi^\prime_s(P_2\square P_n)\geq \chi^\prime_s(P_2\square P_3)>3$. 
Now,  we give a 4-star edge coloring of $P_2\square P_n$, where $n\geq 3$. Consider the edge coloring $f_{2,n}$ of $P_2\square P_n$ as follows. For every $j$, $0\leq j\leq n-2$, let
\[f_{2,n}((i,j)(i,j+1))=\begin{cases}
j \hspace*{-1mm}\pmod{4}&\quad\text{if}\hspace*{2mm}i=0,\\
j+3\hspace*{-1mm}\pmod{4} &\quad\text{if}\hspace*{2mm}i=1.
\end{cases}
\]
For every $j$, $0\leq j\leq n-1$, let
$$f_{2,n}((0,j)(1,j))=j+1 \hspace*{-1mm}\pmod{4}.$$
Since $f_{2,n}$ has a repeating pattern, to see that  there is  no bi-colored 4-path (4-cycle)  in $P_2\square P_n$, it suffices to check that $f_{2,7}$ is a star edge coloring. The edge  coloring $f_{2,7}$ is shown in Figure \rm{\ref{fig2}}, that is  clearly a 4-star edge coloring. Therefore,  for every $n\geq 3$, $\chi^\prime_s(P_2\square P_n)=4$.
 \begin{figure}[!ht]\centering
\includegraphics[scale=.9]{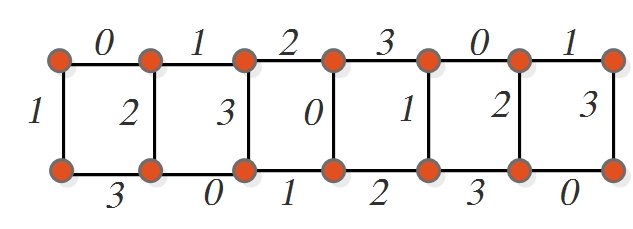}
\caption{4-star edge colorings of  $P_2\square P_7$.}
\label{fig2}
\end{figure}
\par \noindent\textbf{Case 2: }$m\in\lbrace3,4\rbrace$ and $n=3$.\\
By checking all possibilities, it can be seen that there is no 4-star edge coloring of $P_3\square P_3$. Therefore, $\chi^\prime_s(P_3\square P_3)>4$. In Figure \rm{\ref{fig4_3a}}, a 5-star edge coloring of $P_3\square P_3$ is presented. Thus,  $\chi^\prime_s(P_3\square P_3)=5$. Since $\chi^\prime_s(P_4\square P_3)\geq\chi^\prime_s(P_3\square P_3)= 5$ and in Figure \rm{\ref{fig4_3b}} a 5-star edge coloring of $P_4\square P_3$ is presented, we conclude   $\chi^\prime_s(P_4\square P_3)=5$.
 \begin{figure}[!ht]\centering
\subfigure[$P_3\square P_3$]{\includegraphics[scale=.85]{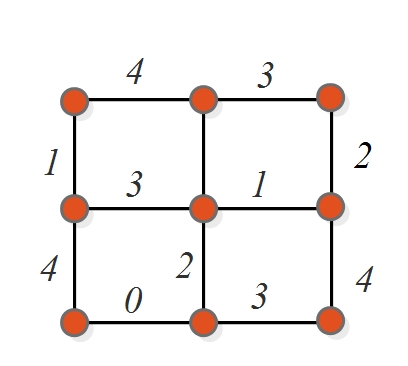}\label{fig4_3a}}
\subfigure[$P_4\square P_3$ ] {\includegraphics[scale=.9]{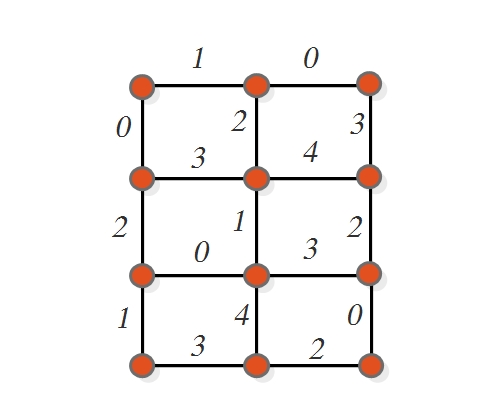}\label{fig4_3b}}
\caption{5-star edge colorings of $P_3\square P_3$ and $P_4\square P_3$.}
\end{figure}
%
%
\par 
\noindent\textbf{Case 3:}  $m=n=4$ or ($m\geq 5$ and $n\geq 3$).\\
In this case,  we first show that $\chi^\prime_s(P_m\square P_n)\geq 6$. For this purpose, we investigate all possible 5-star edge coloring of $P_4\square P_3$ and then we show it is impossible to extend these edge colorings to a 5-star edge coloring of $P_m\square P_n$, when $m=n=4$ or $m\geq 5$, $n\geq 3$. Consider path $P:(1,0)(0,0)(0,1)(0,2)$ in $P_4\square P_3$. 
In a 5-star edge coloring of this graph, edges $(1,0)(0,0)$  and $(0,1)(0,2)$ are either have the same color or not. It can be checked  that in each case, there is only one 5-star edge coloring of $P_4\square P_3$, as shown in  Figure~\rm{\ref{fig8-56}}.
\begin{figure}[!ht]\centering
\subfigure[]{\includegraphics[scale=.85]{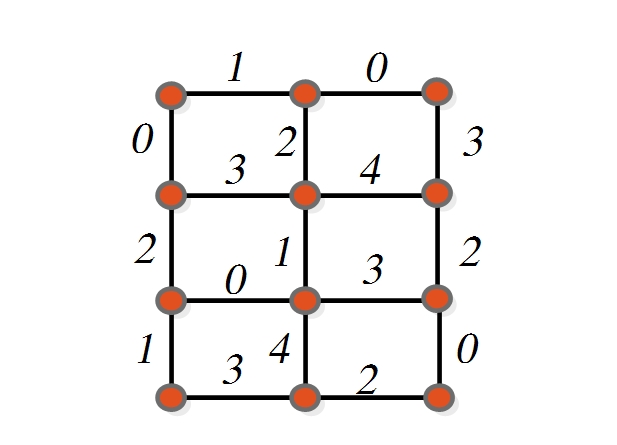}\label{fig8-6-2}}
\subfigure[] {\includegraphics[scale=.85]{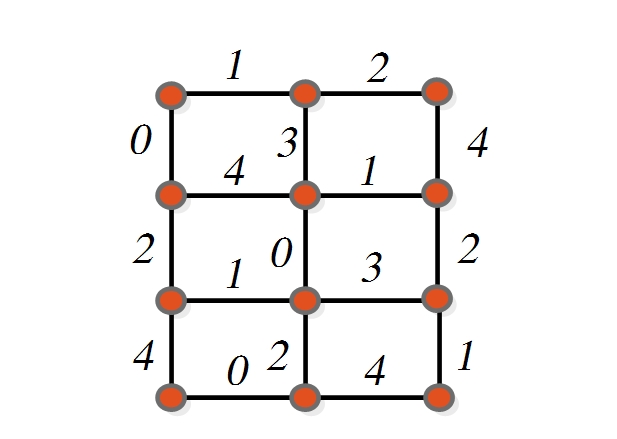}\label{fig8-5-2}}
\caption{All possibilities for having a 5-star edge coloring of $P_4\square P_3$.}
\label{fig8-56}
\end{figure}
\par 
It should be note that by applying permutation $(4,0,1,3)$ on the colors of the 5-star edge coloring of $P_4\square P_3$ in Figure \ref{fig8-5-2},  there exists an automorphism of $P_4\square P_3$  that preserves the edge colors. Hence, we conclude that up to isomorphism, there exists a unique 5-star edge coloring of $P_4\square P_3$.
\par  Now in order to have a 5-star edge coloring of graph $P_m\square  P_n$, when $m=n=4$, or $m\geq5$ and $n\geq 3$, we try to  extend the coloring of $P_4\square P_3$ to a 5-star edge coloring of $P_5\square P_3$ or $P_4\square P_4$. In Figure~\rm{\ref{fig6}}, all possibilities to obtain the  desired colorings are investigated. It turns out that it is impossible to have such a coloring and therefore in this case $\chi^\prime_s(P_m\square P_n)>5$. 
\begin{figure}[!ht]\
\begin{center}
\includegraphics[scale=.8]{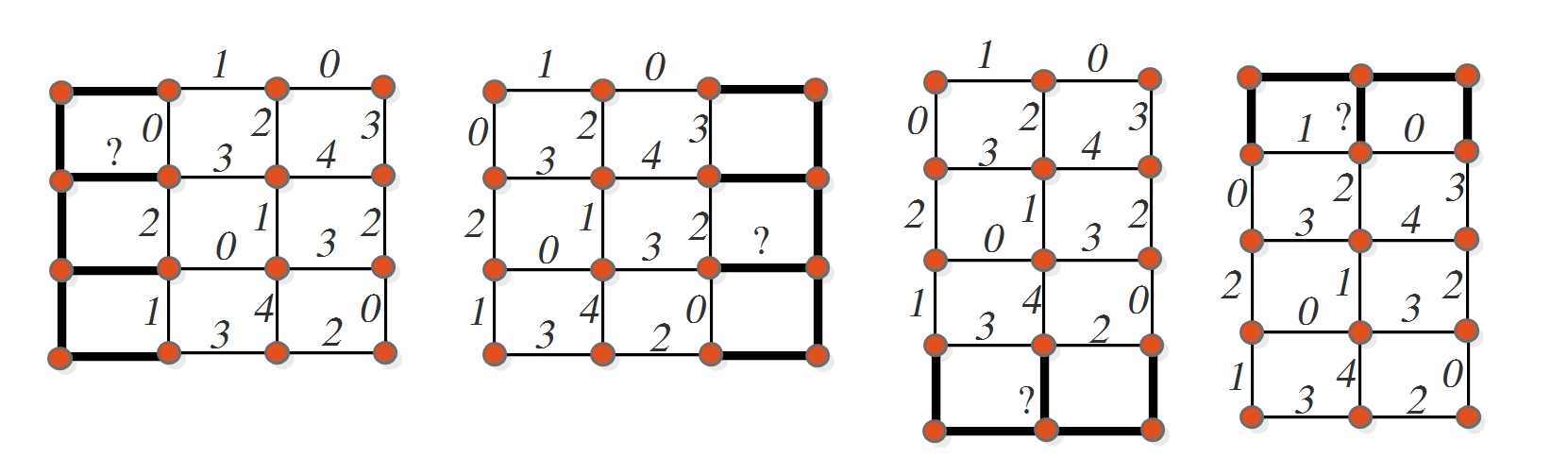}\\
\caption{There is no 5-star edge coloring of $P_4\square P_4$ and $P_5\square P_3$.}
\label{fig6}
\end{center}
\end{figure}
\par 
We now define edge coloring $f_{m,n}:E(P_m\square P_n)\rightarrow\lbrace 0,1,\ldots,5\rbrace$  as follows. For every  $i$ and  $j$, where    $0\leq i\leq m-2$, and $0\leq j\leq n-1$, we have 
\[f_{m,n}((i,j)(i+1,j))=
\begin{cases}
i \hspace*{-1.5mm} \pmod{4}&\quad\text{if}\hspace*{2mm} j=0\hspace*{-1.5mm} \pmod{2},\\
i+3 \hspace*{-1.5mm} \pmod{4}&\quad\text{if}\hspace*{2mm} j=1\hspace*{-1.5mm} \pmod{2}.
\end{cases}
\]
For every  $i$ and  $j$, where    $0\leq i\leq m-1$, and $0\leq j\leq n-2$, we have 
\[f_{m,n}((i,j)(i,j+1))=
\begin{cases}
4+(i\hspace*{-1.5mm} \pmod{2})&\quad\text{if}\hspace*{2mm} j=1 \hspace*{-1.5mm} \pmod{4},\\
5-(i\hspace*{-1.5mm} \pmod{2})&\quad\text{if}\hspace*{2mm} j=3 \hspace*{-1.5mm} \pmod{4},\\
i+1\hspace*{-1.5mm}\pmod{4}&\quad\text{otherwise.}
\end{cases}
\]
Since  $f_{m,n}$ has a repeating pattern, it suffices to check that $f_{7,6}$ is a 6-star edge coloring. The edge coloring $f_{7,6}$  is shown in Figure \ref{fig7}; we can see that there is no bi-colored 4-path (4-cycle) in $P_7\square P_6$. 
 \begin{figure}[!ht]\
\begin{center}
\includegraphics[scale=.8]{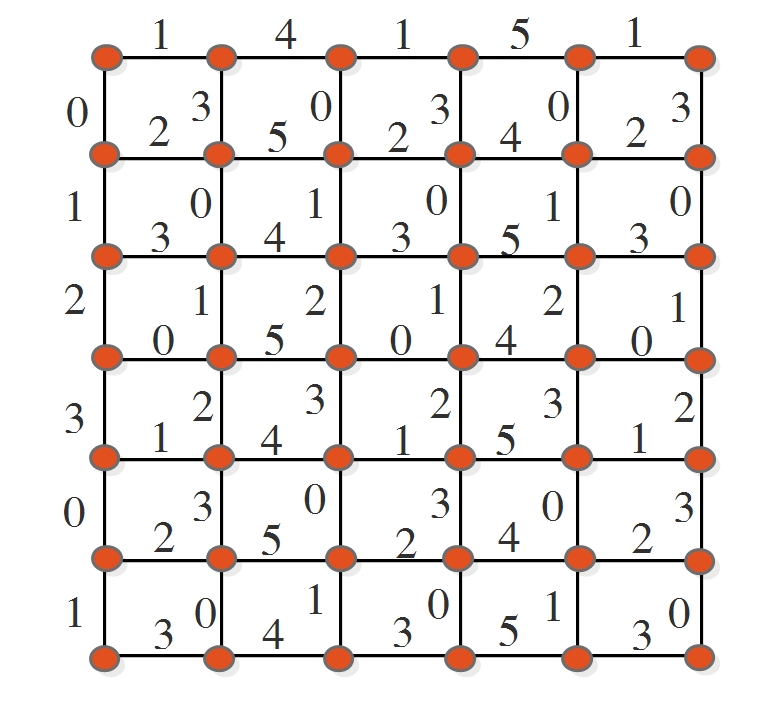}\\
\caption{A 6-star edge coloring  of $P_7\square P_6$.}
\label{fig7}
\end{center}
\end{figure}
}
\end{proof}
By  Corollary \rm{\ref{cor:Gd}}, Theorem \ref{th:col}, \rm{\ref{per:path}}, and  \ref{th4}, we can obtain an upper bound on the star chromatic index of $d$-dimensional grids as follows.

\begin{corollary}\label{cor1}
 If $G_{l_1,l_2,\ldots,l_d}$ is a d-dimensional grid, where $d\geq  2$, then
\begin{equation*}
\chi^\prime_s(G_d )\leq 4d-2.
\end{equation*}
Moreover,  for $d=2$ and $l_1,l_2\geq 4$, this bound is tight. 
\end{corollary}
\begin{proof}
{ If $d=2$, then by Theorem \rm{\ref{th4}}, $\chi^\prime_s(G_2)=\chi^\prime_s(P_{l_1}\square P_{l_2})\leq 6$. Now, assume that $d>2$.
By  Corollary \rm{\ref{cor:Gd}} and Theorem \rm{\ref{per:path}},  $G_{d-2}$ is $(4(d-2),2)$-star colorable. Thus by Theorem \rm{\ref{th:col}} and 6, we conclude that 
$$\chi^\prime_s(G_d)\leq4(d-2)+\chi^\prime_s(P_{l_{1}}\square P_{l_2})\leq 4d-2.$$
}
\end{proof}

In the following theorem, we consider the star chromatic index of the Cartesian product of paths and cycles.
\begin{theorem}\label{th8}
For path $P_m$ and cycle $C_n$,  we have
\[   
\chi^\prime_s(P_m\square C_n) = 
     \begin{cases}
       4 &\quad\text{if}\hspace*{2mm} m=2,n=4,\\
      5 &\quad\text{if}\hspace*{2mm} m=2,n\geq 5, \\
      6 &\quad\text{if}\hspace*{1.3mm} (m\geq k-1, n= 0\hspace*{-2.5mm}\pmod{k}, k\in \lbrace3,4\rbrace)\hspace{1.3mm}\text{or}\hspace{1.3mm}  (m\in\lbrace 3,4\rbrace, n= 2\hspace*{-2.5mm} \pmod 4) ,\\
          \leq 7 &\quad\text{otherwise.} \ 
     \end{cases}
\]
\end{theorem}
\begin{proof}
{ By Theorem \ref{th4}, for every $n\geq 3$ we have $\chi^\prime_s(P_2\square P_n)=4$. Hence,  $\chi^\prime_s(P_2\square C_n)\geq4$. In Figure~\ref{fig10}, a 4-star edge coloring of $P_2\square C_4$ is shown. Thus, $\chi^\prime_s(P_2\square C_4)=4$.
\begin{figure}[!ht]\
\begin{center}
\includegraphics[scale=.9]{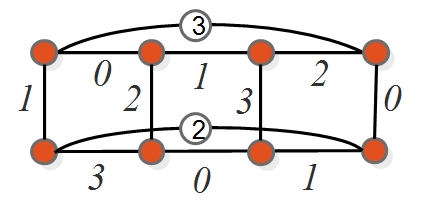}\\
\caption{A 4-star edge coloring of $P_2\square C_4$.}
\label{fig10}
\end{center}
\end{figure}
\par 
As shown in Figure~\ref{fig9}, there are only three  4-star edge colorings of $P_2\square P_3$. It is easy to check that  only the coloring given in Figure \ref{fig9-1}    can be  extended  to a 4-star edge coloring of $P_2\square P_n$, where $n\geq 5$. Infact, the extension of the coloring shown in Figure \ref{fig9-1} is unique and is the edge coloring $f_{2,n}$   in proof of  Theorem~\ref{th4}. But, note that  edge coloring $f_{2,n}$ cannot be extended to a 4-star edge  coloring of $P_m\square C_n$ (see Figure \ref{fig2}). Hence, for every $ n\geq 5$, $\chi^\prime_s(P_2\square C_n)\geq5$.  
\begin{figure}[!ht]\centering
\subfigure[]{\includegraphics[scale=.8]{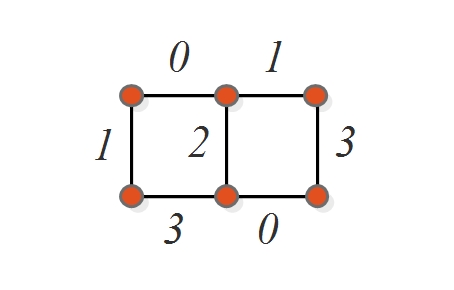}\label{fig9-1}}
\subfigure[] {\includegraphics[scale=.8]{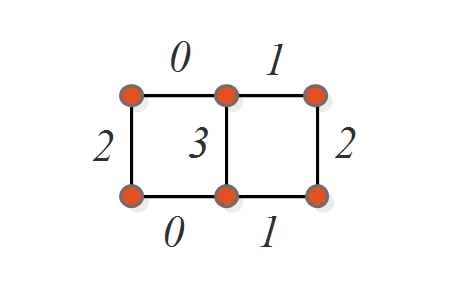}\label{fig9-2}}
\subfigure[] {\includegraphics[scale=.8]{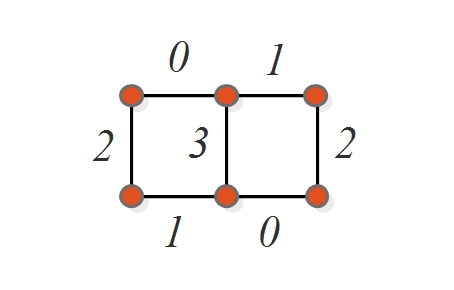}\label{fig9-3}}
\caption{All 4-star edge colorings of $P_2\square P_3$.}\label{fig9}
\end{figure}
\par 
  A 5-star edge coloring of $P_2\square C_5$ is shown in Figure~\ref{fig11}.
 If  $n>5$, then we define edge coloring $g_{2,n}$ for  $P_2\square C_n$ as follows. We first present  a 3-star edge coloring of   two copies of $C_n$ with colors $\lbrace0,1,2\rbrace$ such that  $g_{2,n}((0,0)(0,n-1))\neq g_{2,n}((1,0)(1,n-1))$. The remaining edges are colored by the following rule.  For all $j$, $0\leq j\leq n-1$, we have
\begin{eqnarray*}
  g_{2,n}((0,j)(1,j))=\begin{cases}
       3 &\quad\text{if}\hspace*{2mm} j= 0\hspace*{-1.5mm} \pmod 2,\\
      4 &\quad\text{if}\hspace*{2mm}  j= 1\hspace*{-1.5mm} \pmod 2. \
     \end{cases}
\end{eqnarray*}
Now, we  show that $g_{2,n}$ is a star edge coloring of $P_2\square C_n$, where $n>5$. Clearly, edge coloring $g_{2,n}$ is proper. Since  $g_{2,n}((0,0)(0,n-1))\neq g_{2,n}((1,0)(1,n-1))$,   no bi-colored 4-path (4-cycle) contains  two edges   $e=(0,0)(1,0)$ and $e^\prime=(0,n-1)(1,n-1)$.  For every two distinct vertices $x,y\in\nolinebreak V(C_n)$, consider two edges $e_1=(0,x)(1,x)$ and $e_2=(0,y)(1,y)$, where $\lbrace e_1,e_2\rbrace\neq\lbrace e,e^\prime\rbrace$. If  $g_{2,n}(e_1)=g_{2,n}(e_2)$, then the distance between $e_1$ and $e_2$ is at least two. It implies that it is impossible to get a bi-colored  4-path (4-cycle) under edge coloring $g_{2,n}$.  Thus, edge coloring $g_{2,n}$ is  a 5-star edge coloring of $P_2\square C_n$ and therefore $\chi^\prime_s(P_2\square C_n)=5$, for $n> 5$. 
\begin{figure}[!ht]\
\begin{center}
\includegraphics[scale=.9]{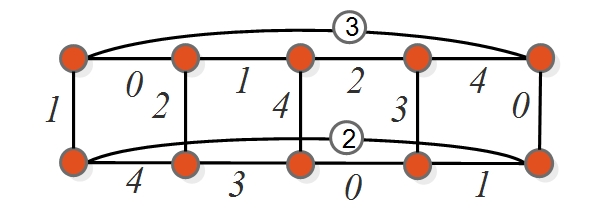}\\
\caption{A 5-star edge coloring of $P_2\square C_5$.}
\label{fig11}
\end{center}
\end{figure}


\par
   
 

 Now, let $m\geq 2$ and $n=0\hspace*{-1mm} \pmod 3$. In Figure \ref{fig10_5}, all possibilities to have a 5-star edge coloring of $P_2\square C_3$ are investigated and it is seen that this is impossible. Therefore, in this case $\chi^\prime_s(P_m\square C_n)\geq6$.  Note that, without loss of generality, in Figure \ref{fig10_5}  we assume  that the path $(1,0)(0,0)(0,1)$ is colored by patterns $(0,1,0)$ or $(0,1,2)$. In this case  we define edge coloring $g_{m,n}$ as follows. For every $i$ and $j$, where  $0\leq i \leq m-2$ and $0\leq j\leq n-1$, we have
 \[ g_{m,n}((i,j)(i+1,j))=
 \begin{cases}
      i \hspace*{-1.5mm} \pmod{3} +3&\quad\text{if}\hspace*{2mm} j= 0\hspace*{-1.5mm} \pmod 3,\\
      i+2 \hspace*{-1.5mm} \pmod{3} +3 &\quad\text{if}\hspace*{2mm}  j= 1\hspace*{-1.5mm} \pmod 3,\\
      i+1 \hspace*{-1.5mm} \pmod{3} +3&\quad\text{if}\hspace*{2mm}  j= 2\hspace*{-1.5mm} \pmod 3.
     \end{cases}
     \]
 For every $i$ and $j$, where  $0\leq i \leq m-1$ and $0\leq j\leq n-1$, we have
  \[ g_{m,n}((i,j)(i,j+1\hspace*{-2mm}\pmod{n}))=
 \begin{cases}
      i\hspace*{-1.5mm} \pmod{3}&\quad\text{if}\hspace*{2mm} j= 0\hspace*{-1.5mm} \pmod 3,\\
      i+1\hspace*{-1.5mm} \pmod{3}&\quad\text{if}\hspace*{2mm}  j= 1\hspace*{-1.5mm} \pmod 3,\\
     i+2\hspace*{-1.5mm} \pmod{3}&\quad\text{if}\hspace*{2mm}  j= 2\hspace*{-1.5mm} \pmod 3.
     \end{cases}
     \]
\begin{figure}[!ht]\
\begin{center}
\includegraphics[scale=.8]{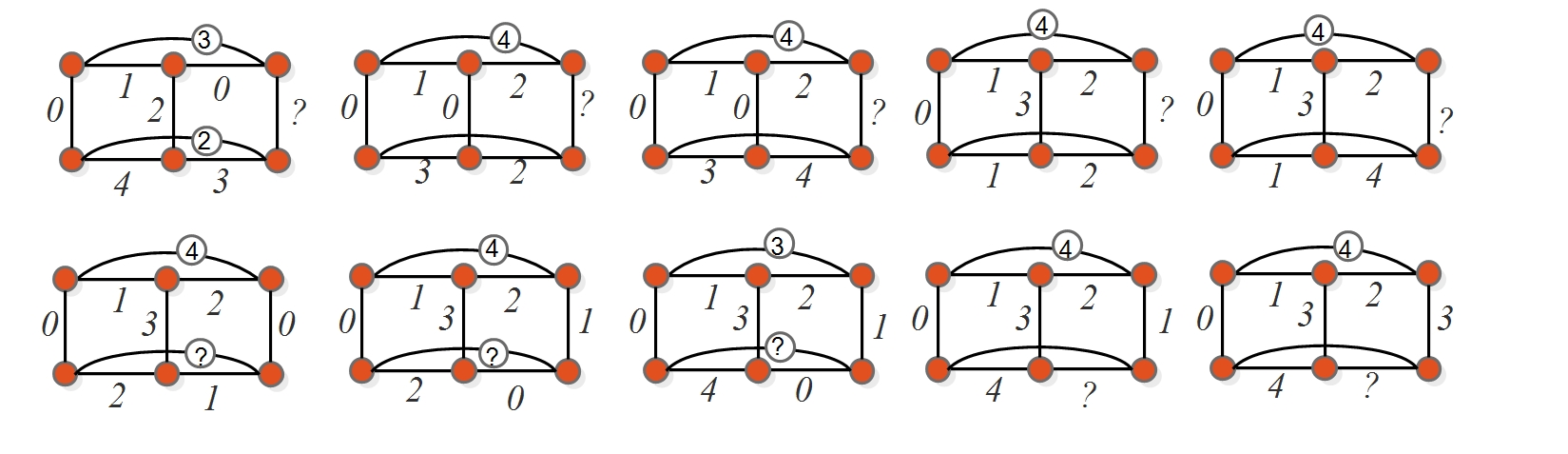}\\
\caption{There is no 5-star edge coloring of $P_2\square C_3$.}
\label{fig10_5}
\end{center}
\end{figure} 

 Since  $g_{m,n}$ has a repeating pattern, it  suffices to check that $g_{5,n}$ is a star edge coloring, for  $n\in\lbrace3,6\rbrace$.
Edge colorings $g_{5,3}$ and $g_{5,6}$ are shown in Figure \ref{fig13}, and it is easy to see that there is no bi-colored 4-path (4-cycle) in $P_5\square C_3$ and $P_5\square C_6$. Thus, we conclude that  for every $m\geq 2$ and $n=0 \hspace*{-1mm} \pmod{3}$, $g_{m,n}$  is a 6-star edge coloring of $P_m\square C_n$.
\begin{figure}[!ht]\centering
\subfigure[$g_{5,3}$]{\includegraphics[scale=.8]{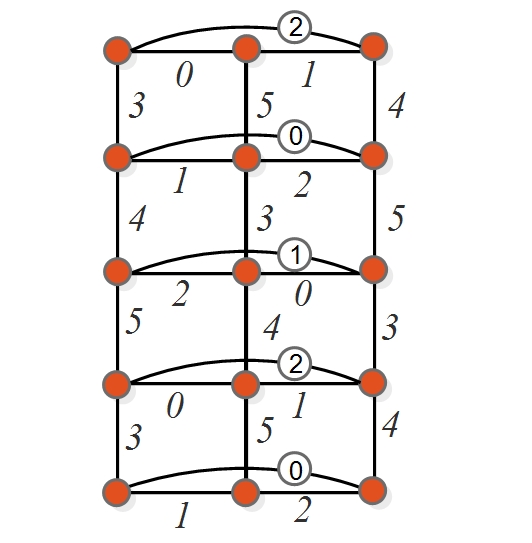}\label{fig13-1}}
\hspace{1mm}
\subfigure[$g_{5,6}$] {\includegraphics[scale=.8]{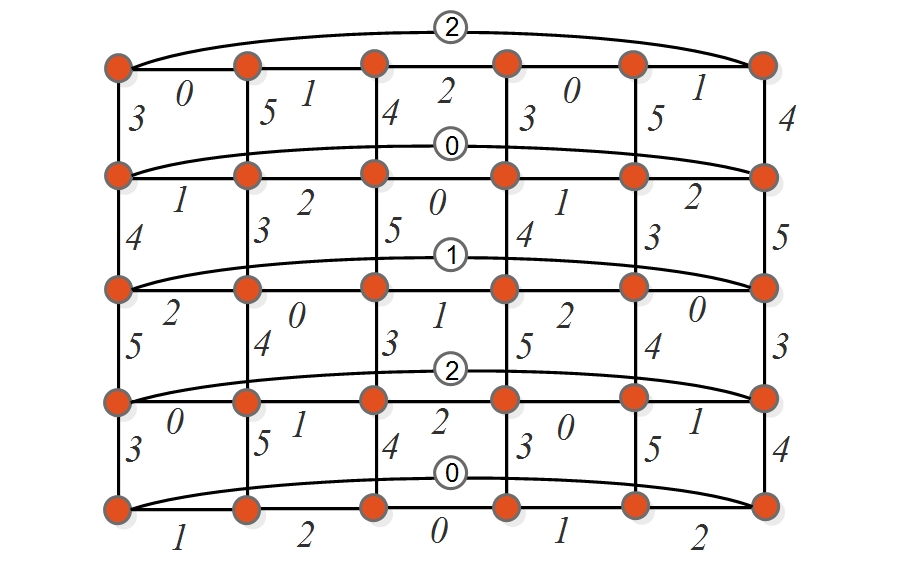}\label{fig13-2}}
\caption{  6-star edge colorings of $P_5\square C_3$ and $P_5\square C_6$.}\label{fig13}
\end{figure}
\par 
If  $m\geq 3$ and $n =0 \hspace*{-1mm} \pmod 4$, we first provide a 6-star edge coloring of $P_m\square C_4$ and then extend it to a 6-star edge coloring of $P_m\square C_n$. As it is shown in proof of  Theorem \ref{th4}, we have only one 5-star edge coloring of $P_3\square P_4$ and this coloring cannot be extended to a 5-star edge coloring of $P_3\square C_4$.  Thus, for every $m\geq 3$, we have
$$\chi^\prime_s(P_m\square C_4)\geq\chi^\prime_s(P_3\square C_4)\geq 6.$$
  By applying edge coloring $f_{m,4}$ in proof of Theorem \ref{th4}, we first color subgraph $P_m\square P_4$ of $P_m\square C_4$. Then, we  color the remaining edges as follows. The  edge sets $\lbrace (i,0)(i,3): i= 1\hspace*{-1mm} \pmod2\rbrace$  and $\lbrace (i,0)(i,3): i= 0\hspace*{-1mm}\pmod2\rbrace$ are colored by 4 and 5, respectively. It is easy to see that this coloring is a 6-star edge coloring of $P_m\square C_4$ and can be repeated, when $n$ is a multiple of 4. Therefore, we  extend obove edge coloring for edge coloring of  $P_m\square C_n$,  where  $m\geq 3$ and $n =0 \hspace*{-1mm} \pmod 4$. 
\par
If $m\in \lbrace 3,4\rbrace$ and $n=2 \hspace*{-1mm} \pmod 4$, then by Theorem~\rm{\ref{per:cycle}}, $C_n$ is $(4,2)$-star colorable. Therefore, by Theorem~\rm{\ref{th:col}}, $\chi^\prime_s(P_m\square C_n)\leq 6$.   Moreover, by Theorem \ref{th4}, we have  
$$\chi^\prime_s(P_m\square C_n)\geq \chi^\prime_s(P_m\square P_n)= 6.$$
 Thus,  in this case $\chi^\prime_s(P_m\square C_n)=6$.
\par 

For other cases, we show that $\chi^\prime_s(P_m\square C_n)\leq 7$.  First assume that $m\geq 5$ and $n=2\hspace*{-1mm} \pmod 4$. By Corollary \ref{cor:comp}, we have 
$$\chi^\prime_s(P_m\square C_n)\leq \chi^\prime_s(C_n)+2\chi(C_n)=7.$$
 If $m\geq 3$ and  $n\in\lbrace1,5\rbrace \hspace*{-1mm} \pmod6$, then by Theorem~\ref{per:cycle},
cycle $C_n$ is (5-2)-star colorable. We use two star compatible colorings $f_0$ and $f_1$ that are presented in proof of  Theorem~\ref{per:cycle} to color copies of $C_n$ as follows. For every $ i\in V(P_m)$ and  $jk\in E(C_m)$, we have
\begin{align*}
g_{m,n}((i,j)(i ,k))=\begin{cases}
f_0(jk)&\text{if} \hspace*{2mm} i= 0 \hspace*{-1.5mm} \pmod{2},\\
f_1(jk)&\text{otherwise.}\
 \end{cases}\
\end{align*}
  Now, we determine  color of the remaining edges, as follows. Let $n=b\hspace*{-1mm}\pmod{4}$, where $b\in\lbrace1,3\rbrace$. Consider  the following ordered $(m-1)$-tuples (each entry in the tuples represents a color).
\begin{align*}
&T_{0,1}=(5,6,5,1,5,6,5,1,\ldots), \hspace*{1mm} T_{1,1}=(5,6,5,2,5,6,5,2\ldots), \hspace*{1mm} T_{2,1}=(5,6,5,4,5,6,5,4,\ldots),\\
  &T_{n-2,1}=(5,6,5,3,5,6,5,3,\ldots),\hspace*{1mm} T_{n-1,1}=(5,6,5,0,5,6,5,0,\ldots),\\
&T_{0,3}=(6,5,3,5,6,5,3,5,\ldots), \hspace*{1mm} T_{1,3}=(5,2,5,6,5,2,5,6,\ldots), \hspace*{1mm}T_{2,3}=(5,4,5,6,5,4,5,6,\ldots),\\
&T_{n-2,3}=(5,1,5,6,5,1,5,6,\ldots), \hspace*{1mm}T_{n-1,3}=(0,5,6,5,0,5,6,5,\ldots).
\end{align*} 
 We denote the $l$-th entry of $T_{i,j}$ by $T_{i,j}^l$.  For every $i$ and $j$, where $0\leq i\leq m-2$ and $0\leq j\leq n-1$, let
 \begin{align*}
g_{m,n}((i,j)(i+1 ,j))=\begin{cases}
T_{2,b}^{i+1}&\text{if} \quad 2\leq j\leq n-3,\\
T_{j,b}^{i+1}& \text{otherwise}.\
 \end{cases}\ 
\end{align*}
%
For $m\geq 3$ and  $n\in\lbrace1,5\rbrace \hspace*{-1mm} \pmod6$, $g_{m,n}$ has a repeating pattern. Therefore,  to see that  $g_{m,n}$ is a star edge coloring of  $P_m\square C_n$, it suffices to check that $g_{7,9}$ and $g_{7,7}$ are star edge colorings. The edge  colorings $g_{7,9}$ and $g_{7,7}$ are shown in Figure~\rm{\ref{figfig1}}, they are  clearly 7-star edge colorings. 
 \begin{figure}[!ht]\centering
\subfigure[$g_{7,9}$]{\includegraphics[scale=.75]{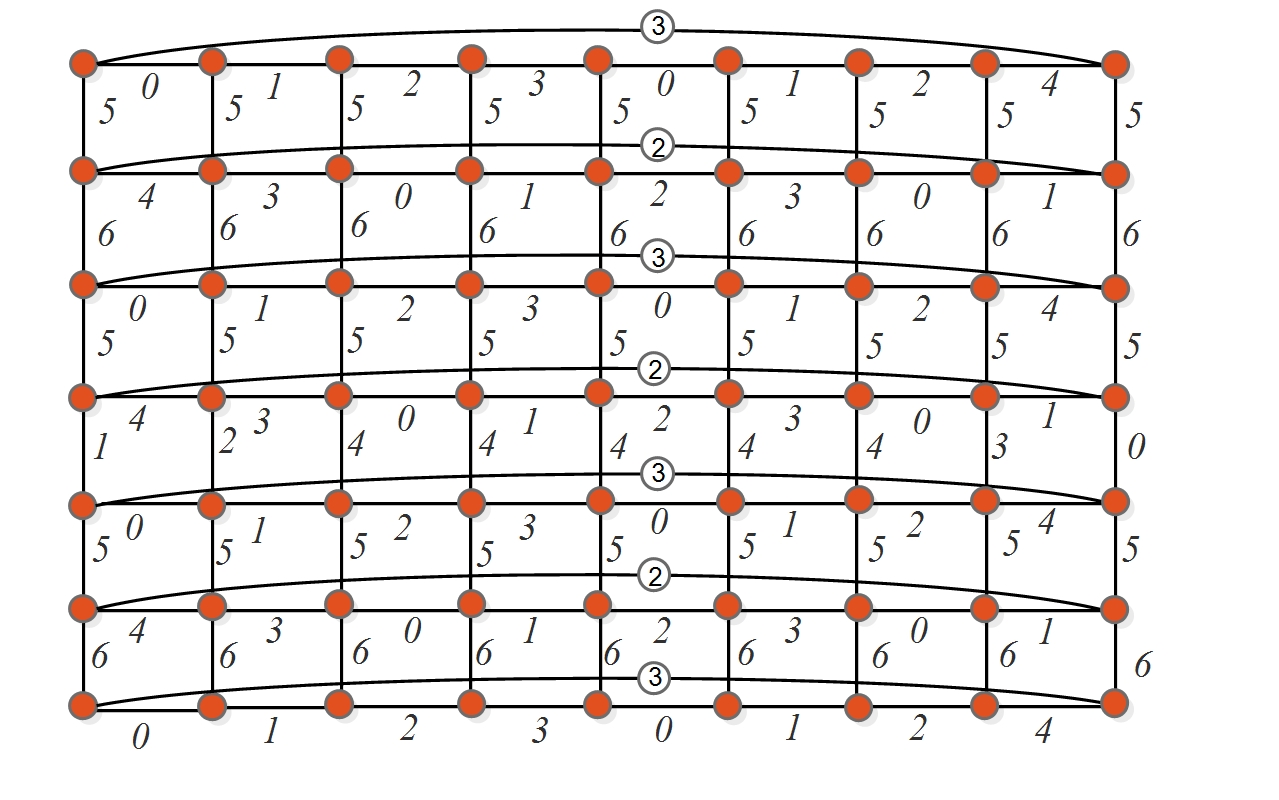}\label{fig14}}
\subfigure[$g_{7,7}$ ] {\includegraphics[scale=.75]{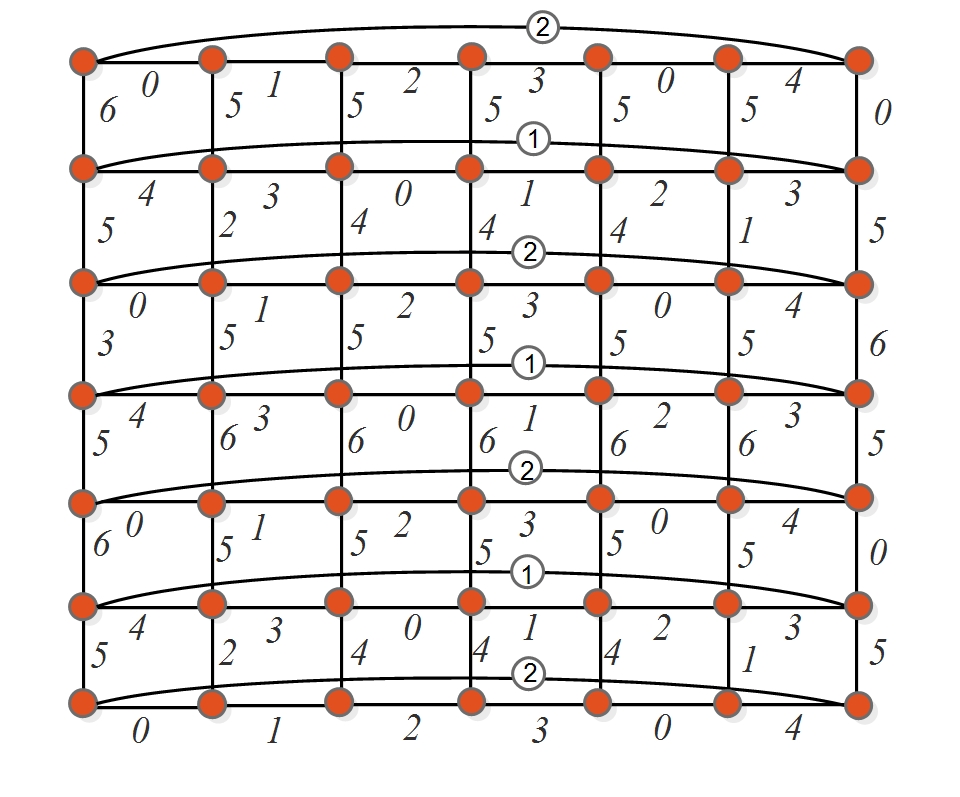}\label{fig15}}
\caption{ 7-star edge colorings of $P_7\square C_9$ and $P_7\square C_7$.}\label{figfig1}
\end{figure}
}
\end{proof}
Now, by Theorem \ref{th2} and \ref{th8}, we obtain an upper bound on the star chromatic index of hypercube $Q_d$, where $d\geq 3$,  as follows.
\begin{corollary}\label{corhyp}
If $Q_d$ is  d-dimensional hypercube and $d\geq 3$, then 
$$\chi^\prime_s(Q_d)\leq 2d-2.$$
Moreover this bound is tight.
\end{corollary}
\begin{proof}{
It is known that  for every integer $d\geq 2$, $Q_{d}=Q_{d-1}\square P_2$ and  $\chi(Q_{d})=2$ \cite{bondy}.  Therefore,  by Theorem \ref{th2}, we have
$$\chi^\prime_s(Q_{d})\leq \chi^\prime_s(P_2)\chi(Q_{d-1})+\chi^\prime_s(Q_{d-1})\leq 2+ \chi^\prime_s(Q_{d-1}).$$
 Also, by Theorem~\ref{th8},  $\chi^\prime_s(Q_3)=\chi^\prime_s(P_2\square C_4)=4$. Thus, by  induction on $d$, it follows that
 $$\chi^\prime_s(Q_d)\leq  2d-2.$$
 Note that $Q_4= C_4\square C_4$. Then, by Theorem~\ref{th8}, we have
 $$\chi^\prime_s(Q_4)=\chi^\prime_s(C_4\square C_4)\geq  \chi^\prime_s(P_4\square C_4)=6,$$
 which implies $\chi^\prime_s(Q_4)=6$. Therefore, for $d=3$ and $d=4$, the upper bound is tight.
}\end{proof}

 In the following theorem, By Theorem \ref{th:col} and \ref{per:cyclei}, we give some upper bounds on the star chromatic index of the Cartesian product of two cycles.
\begin{theorem}\label{th9}
For every integers $m,n\geq 3$, we have the following statements.
\begin{enumerate}
\rm{\item[(i)] } If $m$ and $n$ are even integers, then $\chi^\prime_s(C_{m}\square C_{n})\leq 7$.
\item[(ii)] If $m>3$ is an odd integer and  $n$  is even integer, then $\chi^\prime_s(C_{m}\square C_{n})\leq 8$.
\item[(iii)] If $m=3$ and   $n$ is an even integer, then $\chi^\prime_s(C_{m}\square C_{n})\leq 9$.
\item[(iv)] If $m$ and   $n$ are odd integers, then $\chi^\prime_s(C_{m}\square C_{n})\leq 10$.
\end{enumerate}
\end{theorem}
\begin{proof}
{
\rm{(i)} Let $m$ and $n$ be both even integers. By Proposition~\ref{p1}, $\chi^\prime_s(C_m)=3$ and by Theorem~\ref{per:cyclei}(i), cycle $C_{m}$ is $(4,2)$-star colorable. Thus, by applying Theorem \ref{th:col}, $\chi^\prime_s(C_{m}\square C_{n})\leq 4+\chi^\prime_s(C_{n})=7.$
\par 
\noindent
\rm{(ii)}  Let $m>3$ be an odd integer  and  $n$  be an even integer. By Theorem \ref{per:cycle}(ii), cycle $C_m$ is $(5,2)$-star colorable. Therefore,  by Theorem \ref{th:col}, $\chi^\prime_s(C_m\square C_n)\leq 5+\chi^\prime_s(C_n)=8$, as desired. 
\par  
\noindent
\rm{(iii)}  Let  $m=3$ and $n$ be an even integer. By Theorem \ref{per:cycle}(iii),  $C_3$ is $(6,2)$-star colorable. Then,  by Theorem \ref{th:col}, $\chi^\prime_s(C_3\square C_n)\leq 6+\chi^\prime_s(C_n)=9$. 
\par

\noindent
\rm{(iv)}
Let $m$ and $n$ be  both odd integers. A 6-star edge coloring of $C_3\square C_3$ and a 7-star edge coloring of $C_5\square C_5$  are shown in Figure \ref{figcycle3-3}.
\begin{figure}[!ht]\
\begin{center}
\subfigure[$C_3\square C_3$]{\includegraphics[scale=.77]{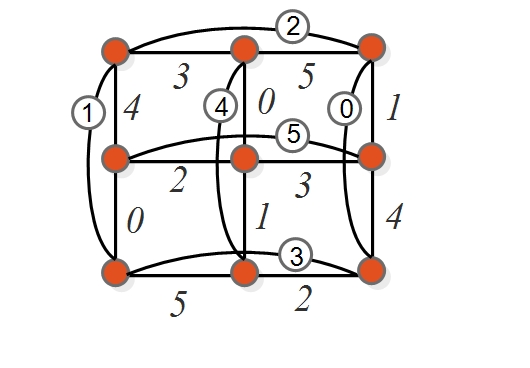}\label{figc3c3}}
\subfigure[$C_5\square C_5$ ] {\includegraphics[scale=.77]{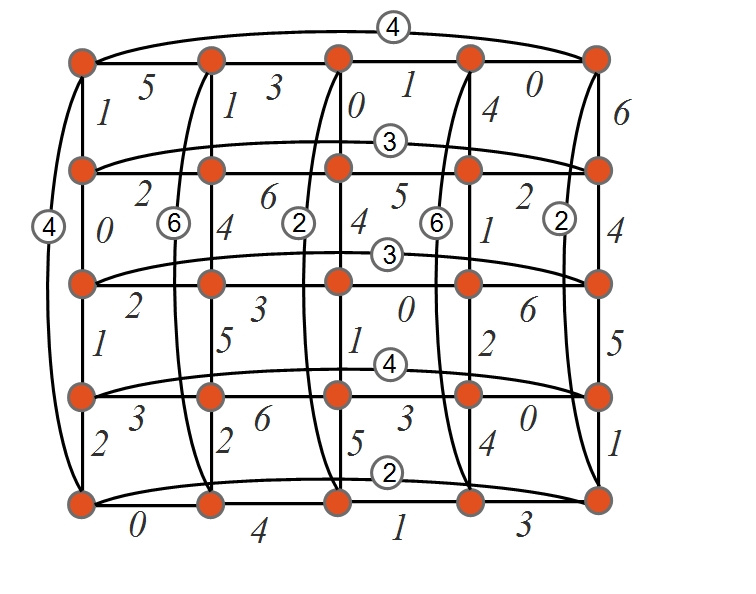}\label{figc5c5}}
\caption{A 6-star edge coloring of $C_3\square C_3$ ans a 5-star edge coloring of $C_5\square C_5$.}
\label{figcycle3-3}
\end{center}
\end{figure}
\par  Now, assume that $m$ and $n$ are not both 3 or 5. Without loss of generality, assume that $m>3$. 
If $m> 5$ and $n=5$, then $\chi^\prime_s(C_m)=3$ and,  by Theorem \ref{per:cycle}(ii), $C_n$ is $(7,3)$-star colorable. Hence, $\chi^\prime_s(C_m\square C_5)\leq 7+\chi^\prime_s(C_m)=10$.
In other cases, $\chi^\prime_s(C_n)=3$ and   $C_m$ is $(7,3)$-star colorable.   Thus, $\chi^\prime_s(C_m\square C_n)\leq 7+\chi^\prime_s(C_n)=10$. 
}
\end{proof}
\textbf{Remark.} By giving some more complex pattern for the  star edge coloring of $C_m\square C_n$, we have found that $\chi^\prime_s(C_m\square C_n)\leq 8$, when $m$ or $n$ is an odd integer. Moreover, we conjectured that $\chi^\prime_s(C_m\square C_n)\leq 7$, for every integers $m$ and $n$. 

%
\par By  Corollary \rm{\ref{cor:Gd}}, Theorem \ref{th:col}, \ref{per:cycle}, and  \ref{th9}, we can obtain an upper bound on the star chromatic index of $d$-dimensional toroidal grids as follows.
\begin{corollary}\label{col:tor}
For every  integers  $d\geq2$ and $l_1,l_2,\ldots,l_d\geq 3$,  we have the following statements.
\begin{enumerate}
\item[\rm{(i)}] The  toroidal grid $T_{2l_1,2l_2,\ldots,2l_d}$ is $(4d,2)$-star colorable and  $\chi^\prime_s(T_{2l_{1},2l_{2},\ldots,2l_d})\leq 4d-1$.
\item[\rm{(ii)}] If  every $l_i>3$, $1\leq i\leq d$,  then  the toroidal grid $T_{l_1,l_2,\ldots,l_d}$ is $(7d,3)$-star colorable and $\chi^\prime_s(T_{l_{1},l_{2},\ldots,l_d})\leq 7d-4$.
\end{enumerate}
\end{corollary}

\begin{proof}
{By Theorem \ref{per:cycle}, every even cycle is $(4,2)$-star colorable and every odd cycle (or every cycle), except $C_3$, is $(7,3)$-star colorable. Then, by Corollary \ref{cor:Gd}, $T_{2l_1,2l_2,\ldots,2l_{d}}$ is $(4d,2)$-star colorable and 
  $T_{l_1,l_2,\ldots,l_{d}}$ is $(7d,3)$-star colorable. Thus, by  Theorem \ref{th:col} and \ref{th9}, we have 
   $$\chi^\prime_s(T_{2l_1,2l_2,\ldots,2l_{d}})\leq 4(d-2)+\chi^\prime_s( C_{2l_{d-1}}\square C_{2l_{d}})\leq 4d-1,$$ 
   and  
   $$\chi^\prime_s(T_{l_1,l_2,\ldots,l_{d}})\leq 7(d-2)+\chi^\prime_s( C_{l_{d-1}}\square C_{l_{d}})\leq 7d-4.$$
}
\end{proof}

\setlength{\baselineskip}{0.8\baselineskip}

\end{document}